\documentclass[12pt]{amsart}
\usepackage{amsmath, amssymb, latexsym, amsthm,bm}
\usepackage{enumitem}
\usepackage{mathrsfs}
\usepackage{fullpage}
\usepackage{color}
\usepackage{tikz}
\usepackage{graphicx}
\usepackage{stmaryrd,amsbsy,enumerate}
\usepackage{mathscinet}
\usetikzlibrary{calc}
\usepackage{cancel}
\usepackage{xspace}
\usepackage{hyperref}
\usepackage{url}
\usepackage{longtable}
\usepackage[obeyFinal,colorinlistoftodos,textsize=tiny]{todonotes}
\usepackage{float}
\usepackage{caption}
\usepackage{amsrefs}
\usepackage{tikz}
\usepackage{caption}
\usepackage{mathtools,amsmath}
\usepackage{array}
\usepackage{stmaryrd}
\usepackage{lineno}

\usepackage{soul}
\usepackage[author={Max Schlepzig}]{pdfcomment}

\usepackage{lmodern}

\def\Tinv{\reflectbox{\rotatebox[origin=c]{180}{$T$}}}
\def\mb#1{{\mathrm{\mathbf{#1}}}}
\def\C{{\mb{C}}}
\def\T{{\mb{T}}}
\def\Cinv#1{{(\mb{C}_{#1})^{-1}}}
\def\Tinv#1{{(\mb{T}_{#1})^{-1}}}

\def\ceil#1{{\lceil{#1}\rceil}}

\newtheorem{theorem}{Theorem} 
\newtheorem{theorem*}{Theorem}

\newtheorem{claim}[theorem]{Claim}

\begin{document}

% *****************************************************************************
\title{\bf The unavoidable rotation systems} 
% *****************************************************************************

% *****************************************************************************
\author{Alan Arroyo}
\address{IST Austria, 3400 Klosterneuburg, Austria}
\email{\tt alanmarcelo.arroyoguevara@ist.ac.at}

\author{R.~Bruce Richter}
\address{Department of Combinatorics and Optimization. University of Waterloo. Waterloo, ON Canada N2L 3G1}
\email{\tt brichter@uwaterloo.ca}

\author{Gelasio Salazar}
\address{Instituto de F\'\i sica, Universidad Aut\'onoma de San Luis Potos\'{\i}, SLP 78000, Mexico}
\email{\tt gsalazar@ifisica.uaslp.mx}

\author{Matthew Sullivan}
\address{Department of Combinatorics and Optimization. University of Waterloo. Waterloo, ON Canada N2L 3G1}
\email{\tt m8sullivan@uwaterloo.ca}
% *****************************************************************************

% *****************************************************************************
\begin{abstract}
For each positive integer $m$, Pach, Solymosi, and T\'oth identified two canonical complete topological subgraphs $C_m$ and $T_m$, and proved that every sufficiently large topological complete graph contains $C_m$ or $T_m$ as a subgraph. We generalize this result in the setting of abstract rotation systems.
\end{abstract}
% *****************************************************************************

\maketitle

% **************************************************************

% *****************************************************************************
% *****************************************************************************
% ***************************   INTRODUCTION **********************************
% *****************************************************************************
% *****************************************************************************

\section{Introduction}\label{sec:intro}

This note is motivated by a work in which Pach, Solymosi, and T\'oth identify the ``unavoidable'' complete topological graphs. We recall that a {\em topological graph} is a graph drawn in the plane, so that every pair of edges intersect each other in at most one point, which is either a common endvertex or a crossing. Two topological graphs $G=(V(G),E(G))$ and $H=(V(H),E(H))$ are {\em weakly isomorphic} if there is an incidence preserving one-to-one correspondence between $G$ and $H$ such that two edges of $G$ intersect if and only if the corresponding edges of $H$ intersect.

As a generalization of the Erd\H{o}s-Szekeres theorem on point sets in general position in the plane~\cite{es2}, Pach, Solymosi, and T\'oth~\cite{pst} proved that the {\em convex graph} $C_m$ and the {\em twisted graph} $T_m$~\cite{HM} (the obvious generalizations, for arbitrary $m$, of the graphs shown in Figure~\ref{fig:c6t6}) are the {\em unavoidable} complete topological graphs:

% *****************************************************************************
\begin{theorem}[{\cite[Theorem 2]{pst}}]\label{thm:main2}
For each integer $m\ge 1$, there is an integer $n$ with the following property. Every complete topological graph with at least $n$ vertices has a complete topological subgraph weakly isomorphic to $C_m$ or $T_m$.
\end{theorem}
% *****************************************************************************

% *****************************************************************************
\begin{figure}[ht!]
\centering
\scalebox{0.41}{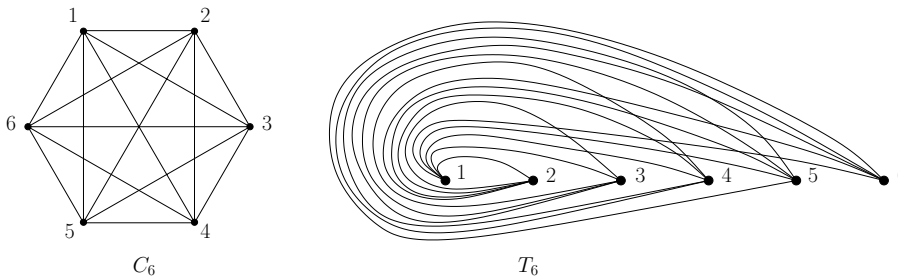}
\caption{A convex complete graph $C_6$, and a twisted complete topological graph $T_6$.}
\label{fig:c6t6}
\end{figure}
% *****************************************************************************

In this note we show that this result still holds in the setting of (abstract) rotation systems. As we note in Section~\ref{sec:remarks}, this generalization allows us to give an alternative proof of Theorem~\ref{thm:main2}, although with a weaker bound that the one given in~\cite{pst}.

% *****************************************************************************
% ******************  SUBSECTION. ROTATION SYSTEMS ****************************
% *****************************************************************************

\subsection{Rotation systems}

Let $G$ be a complete topological graph whose vertices are ordered and labelled $i_1,\ldots,i_n$. The {\em rotation system} $\mb{G}$ of $G$ is the $n$-tuple $\bigl(\mb{G}(i_1),\ldots,\mb{G}(i_n)\bigr)$, where $\mb{G}(i_j)$ is the cyclic permutation of $\{i_1,\ldots,i_n\}{\setminus}\{i_j\}$ that records the clockwise cyclic order of the edges incident with vertex $i_j$.

To illustrate the notion of a rotation system, we refer the reader back to Figure~\ref{fig:c6t6}. It is straightforward to see from these illustrations that, for each integer $m\ge 2$:

\begin{itemize}
\item the vertices of $C_m$ can be labelled $1,\ldots,m$ so that the $i$-th entry of the rotation system $\mb{C}_m$ is $\mb{C}_m(i)= 1\, 2\cdots (i-1)\ (i+1)\cdots (m-1)\, m$ for $i=1,\ldots,m$; and
\item the vertices of $T_m$ can be labelled $1,\ldots,m$ so that the $i$-th entry of the rotation system $\mb{T}_m$ is $\mb{T}_m(i)=n(n-1)\cdots (i+1)\, 1\, 2 \cdots (i-1)$ for $i=1,\ldots,m$.
\end{itemize}

For each labelling of the vertices of a complete topological graph $G$ we obtain a rotation system $\mb{G}$, but the converse is not true: there exist {\em abstract} rotation systems that are not the rotation system of any complete topological graph.

An {\em abstract rotation system} $\Pi$ on an ordered set $\{i_1,\ldots,i_n\}$ is an $n$-tuple $\bigl(\Pi(i_1),\ldots,\Pi(i_n)\bigr)$, where $\Pi(i_j)$ is a cyclic permutation of $\{i_1,\ldots,i_n\}{\setminus}\{i_j\}$ for $j{=}1,\ldots,n$~\cite{kyncl1}. We say that $\{i_1,\ldots,i_n\}$ is the {\em ground set} of $\Pi$, and for convenience without loss of generality we implicitly assume that $i_1,\ldots,i_n$ (the {\em ground elements} of $\Pi$) are integers such that $1\le i_1 < \cdots < i_n$. 

An abstract rotation system $\Pi$ is {\em realizable} if there is a complete topological graph $G$ such that $\mb{G}=\Pi$. For instance, it is easy to check that the abstract rotation system $\bigl( 2 3 4, 1 3 4, 1 2 4,  1 3 2\bigr)$ is not realizable. In what follows, whenever we speak of a rotation system $\Pi$ we implicitly mean that $\Pi$ is an abstract rotation system. %which may or may not be realizable. 

In order to state our main result, let us recall three basic notions on rotation systems. Two rotation systems $\Pi$ and $\Pi'$ are {\em equivalent} if there is a relabelling of the ground elements of $\Pi'$ that turns $\Pi'$ into $\Pi$. If $\Pi=\bigl(\Pi(i_1),\ldots,\Pi(i_n)\bigr)$ is a rotation system, then $\Pi^{-1}$ denotes its {\em inverse} rotation system $\bigl((\Pi(i_1))^{-1},$ $\ldots,(\Pi(i_n))^{-1}\bigr)$, where $(\Pi(i_j))^{-1}$ denotes the inverse of the cyclic permutation $\Pi(i_j)$, for $j=1,\ldots,n$. Finally, given a rotation system $\Pi$, every subcollection of the ground set of $\Pi$ naturally inherits from $\Pi$ a rotation system $\Pi'$, which is a {\em rotation subsystem} of $\Pi$.

%Theorem~\ref{thm:main2} shows the existence of one of $C_m$ and $T_m$ as a subgraph of every sufficiently large complete topological graph. Our main result shows the existence of one of $\C_m,\T_m,\Cinv{m}$, and $\Tinv{m}$ as a subsystem of every sufficiently large rotation system. 

%Theorem~\ref{thm:main2} involves certain complete topological graphs (namely $C_m$ and $T_m$) as subgraphs of complete topological graphs. Our main result involves certain rotation systems (namely $\mb{C}_m, \mb{T}_m, \Cinv{m}$, and $\Tinv{m}$) as {\em subsystems} of rotation systems. 

%Let $\Pi$ be a rotation system of size $n$. Any collection of integers $1\le i_1 < \cdots < i_m \le n$ naturally induces a rotation {\em subsystem} of $\Pi$, as follows. First, discard from $\Pi$ the entries $\Pi(i_1),\ldots,\Pi(i_m)$. Next, for each $\Pi(i_j)$ with $j=1,\ldots,m$ remove each entry labelled with an element in $[n]{\setminus}\{i_1,\ldots,i_m\}$. Finally, perform the relabellings $i_j\mapsto j$, for $j=1,\ldots,m$. As a result we obtain a rotation system on its own right: this is a rotation {\em subsystem} of $\Pi$ of size $m$ or, equivalently, an {\em $m$-subsystem} of $\Pi$. %\redit{Example?}

% *****************************************************************************
% *******************  SUBSECTION. MAIN RESULT  *******************************
% *****************************************************************************

\subsection{Our main result}

The setting of abstract rotation systems is much wider than the setting of realizable rotation systems (that is, the setting of complete topological graphs): the probability that a random abstract rotation system of size $n$ is realizable goes to $0$ as $n$ goes to infinity. Our main result is that Theorem~\ref{thm:main2} still holds in this abstract setting, and $\mb{C}_m, \mb{T}_m, (\mb{C}_m)^{-1}$, or $(\mb{T}_m)^{-1}$ are precisely the unavoidable rotation systems:

%Indeed, it is not difficult to verify that 

% *****************************************************************************
% ******************  Statement of Main Theorem  ******************************
% *****************************************************************************

\begin{theorem}\label{thm:main}
For each integer $m\ge 1$ there is an integer $n_0(m)$ such that every rotation system of size at least $n_0(m)$ has a subsystem equivalent to $\mb{C}_m, \mb{T}_m, (\mb{C}_m)^{-1}$, or $(\mb{T}_m)^{-1}$.
\end{theorem}

% *****************************************************************************
% *****************************************************************************
% *****************************************************************************

% *****************************************************************************
% *****************************************************************************
% ***********************  PROOF OF MAIN THEOREM  *****************************
% *****************************************************************************
% *****************************************************************************

\section{Proof of Theorem~\ref{thm:main}}

Let $\Pi$ be a rotation system with ground set $\{i_1,\ldots,i_n\}$, and let $j\in\{1,\ldots,n\}$. We say that $\Pi(i_j)$ is {\em separated} if it can be written as $(\sigma\tau)$, where $\sigma$ is a permutation of $i_1\cdots i_{j-1}$ and $\tau$ is a permutation of $i_{j+1}\cdots i_n$. Suppose that $\Pi(i_j)$ is separated. We say that $\Pi(i_j)$ is {\em backward increasing} (respectively, {\em backward decreasing}) if $\sigma=i_{1}\cdots i_{j-1}$ (respectively, $\sigma=i_{j-1}\cdots i_{1}$). In either of these cases, $\Pi(i_j)$ is {\em backward monotone}. We say that $\Pi(i_j)$ is {\em forward increasing} (respectively, {\em forward decreasing}) if $\tau=i_{j+1}\cdots i_n$ (respectively, $\tau=i_n\cdots i_{j+1}$). In either of these cases, $\Pi(i_j)$ is {\em forward monotone}.

If $\Pi(i_j)$ is separated for each $j=1,\ldots,n$, then $\Pi$ is {\em separated}. We similarly extend the notions of forward increasing/decreasing, forward monotone, backward increasing/decreasing, and backward monotone, to the whole rotation system.

%its subpermutation that contains $i_{j+1},\ldots,i_m$ is $i_{j+1}\cdots i_m$ (respectively, $i_m\cdots i_{j+1}$). We say that $\sigma$ is {\em backward increasing} (respectively, {\em backward decreasing}) if its subpermutation that contains $i_{1},\ldots,i_{j-1}$ is $i_1 \cdots i_{j-1}$ (respectively, $i_{j-1}\cdots i_1$). We say that $\sigma$ is {\em forward monotone} if it is forward increasing or forward decreasing, and {\em backward monotone} if it is backward increasing or backward decreasing.

%Let $\Pi$ be a rotation system with ground set $i_1,\ldots,i_m$, and let $s\le m$ be a nonnegative integer. We say that $\Pi$ is $s$-\emph{separated} if $\Pi(i_\ell)$ is separated for each $\ell=1,\ldots,s$. In the particular case $s=m$, we simply say that $\Pi$ is \emph{separated}. We say that $\Pi$ is $s$-{\em forward monotone} (respectively, $s$-\emph{backward monotone}) if $\Pi(i_\ell)$ is forward monotone (respectively, backward monotone) for each $\ell=1,\ldots,s$. In the particular case $s=m$, we simply say that $\Pi$ is \emph{forward monotone} (respectively, \emph{backward monotone}). Note that in all these definitions we admit the possibility that $s=0$, in which case there is no condition imposed on $\Pi$.

As we shall see shortly, Theorem~\ref{thm:main} is an easy consequence of the following two statements.

\begin{claim}\label{cla:sieve0}
For each positive integer $t$ there is an integer $n_1(t)$ with the following property. Every rotation system of size at least $n_1(t)$ has a rotation subsystem of size $t$ that is separated. 
\end{claim}

\begin{claim}\label{cla:sieve1}
For each positive integer $t$ there is an integer $n_2(t)$ with the following property. Every separated rotation system of size at least $n_2(t)$ has a rotation subsystem of size $t$ that is forward monotone, and a rotation subsystem of size $t$ that is backward monotone. 
\end{claim}

\begin{proof}[Proof of Theorem~\ref{thm:main}]
Let $m\ge 1$ be an integer. Let $\Pi$ be a rotation system of size at least $n_1(n_2(n_2(4m)))$, where $n_1$ and $n_2$ are as in Claims~\ref{cla:sieve0} and~\ref{cla:sieve1}. By these claims, $\Pi$ has a rotation subsystem $\Pi'$ of size $4m$ that is separated, forward monotone, and backward monotone. By the pigeon-hole principle, $\Pi'$ has a rotation subsystem $\Pi''$ of size $m$ that is separated and either (i) forward increasing and backward increasing; or (ii) forward decreasing and backward increasing; or (iii) forward decreasing and backward decreasing; or (iv) forward decreasing and backward increasing. These four possibilities correspond to $\Pi''$ being equivalent to $\C_m,\T_m,\Cinv{m}$, and $\Tinv{m}$, respectively.
\end{proof}

\begin{proof}[Proof of Claim~\ref{cla:sieve0}]
{Let $t$ be a positive integer. As we shall argue shortly, the claim is an easy consequence of the following.}

\medskip
\noindent ($*$)\,\, \emph{Let $\Pi$ be a rotation system with ground set $\{i_1,\ldots,i_n\}$. Suppose that $\Pi(i_1),\ldots,\Pi(i_s)$ are separated, for some $s\ge 1$. Then $\Pi$ has a subsystem $\Pi'$ of size ${(s{+}1){+}\ceil{(n{-}(s{+}1))/s}}$ that includes $i_1,\ldots,i_s,i_{s+1}$, and $\Pi'(i_1),\ldots,\Pi'(i_{s+1})$ are separated.}

\medskip

{To see that the claim follows from ($*$), let $\Pi$ be a rotation system with ground set $\{j_1,\ldots,j_n\}$. Since $j_1<j_k$ for each $k=2,\ldots,n$, evidently $\Pi(j_1)$ is separated. If $n$ is sufficiently large compared to $t$, then we can iteratively apply ($*$) $t-1$ times, and end up with a rotation subsystem of $\Pi$, of size $t$, that is separated.}

%Let $t$ be a positive integer, and let $\Pi$ be a rotation system with ground set $\{i_1,\ldots,i_n\}$. The key statement in the proof of the claim is the following.

%The claim follows easily from ($*$). To see this, we start by noting that since $i_1 < i_j$ for each $j=2,\ldots,n$, evidently $\Pi(i_1)$ is separated. Now if $n$ is sufficiently large compared to $t$, then we can iteratively apply ($*$) $t-1$ times, and end up with a rotation subsystem of size $t$ that is separated.

To prove ($*$), note that $\Pi(i_{s+1})$ has a substring $\sigma$ of size at least $\ceil{(n-(s+1))/s}$ all of whose entries are in $\{i_{s+2},\ldots,i_n\}$. Let $\Pi'$ be the rotation subsystem of $\Pi$ induced by $i_1,\ldots,i_s,i_{s+1}$ and the entries of $\sigma$. By construction, $\Pi'(i_{s+1})$ is separated, and the separateness of $\Pi'(i_1),\ldots,\Pi'(i_s)$ is inherited from the separateness of $\Pi(i_1),\ldots,\Pi(i_s)$.
\end{proof}

\begin{proof}[Proof of Claim~\ref{cla:sieve1}]
The claim follows by iteratively applying the following statement, and the corresponding statement for the backward monotone case, whose proof is totally analogous.

\medskip
\noindent ($**$)\,\, \emph{Let $\Pi$ be a separated rotation system with ground set $\{i_1,\ldots,i_n\}$, where $\Pi(i_1),\ldots,\Pi(i_s)$ are forward monotone, for some $s\ge 0$. Then $\Pi$ has a subsystem $\Pi'$ of size $(s{+}1){+}\ceil{\sqrt{n{-}s{-}1}}$ whose ground set includes $i_1,\ldots,i_s,i_{s+1}$, and $\Pi'(i_1),\ldots,\Pi'(i_{s+1})$ are forward monotone.}

\medskip

We emphasize the possibility that $s=0$, in which case the hypothesis that $\Pi(i_1),\ldots,\Pi(i_s)$ are forward monotone is vacuous. 

To prove ($**$) we start by noting that the assumption that $\Pi$ is separated implies that $\Pi(i_{s+1})$ can be written as $(\sigma\tau)$, where $\sigma$ is a permutation of $i_1\cdots i_s$ and $\tau$ is a permutation of $i_{s+2} \cdots i_{n}$. Since $\tau$ has size $n-s-1$, by the Erd\H{o}s-Szekeres theorem~\cite{es1} it has a {monotone subpermutation} $\tau'$ of size $\ceil{\sqrt{n-s-1}}$. Let $\Pi'$ be the rotation subsystem of $\Pi$ induced by $i_1,\ldots,i_s,i_{s+1}$ and the entries of $\tau'$. The forward monotonicity of $\Pi'(i_{s+1})$ follows since $\Pi'(i_{s+1})=(\sigma \tau')$ and $\tau'$ is monotone, and the forward monotonicity of $\Pi'(i_1),\ldots,\Pi'(i_s)$ is inherited from the forward monotonicity of $\Pi(i_1),\ldots,\Pi(i_s)$.
%The forward monotonicity of $\Pi'(1),\ldots,\Pi'(s)$ follows from the forward monotonicity of $\Pi(1),\ldots,\Pi(s)$, and 
%The hypothesis is that there are integers $1\le i_1 < \cdots < i_s < \cdots < j_t \le n$ such that for each $\ell=1,\ldots,s$, the restriction of $\Pi(j_\ell)$ to $\{j_1,\ldots,j_t\}\setminus\{j_\ell\}$ is forward monotone. Then there exist $1\le k_1 < \cdots < k_s < k_{s+1} < \cdots k_{s+\ceil{\sqrt{t-s-1}}}$ such that for each $\ell=1,\ldots,s+1$, the restriction of $\Pi(k_\ell)$ to $\{k_1,\ldots,k_{s+1+\ceil{\sqrt{t-s-1}}}\}\setminus\{k_\ell\}$ is forward monotone.
\end{proof}

% *****************************************************************************
% *****************************************************************************
% ************************  CONCLUDING REMARKS  *******************************
% *****************************************************************************
% *****************************************************************************

\section{Concluding remarks}\label{sec:remarks}

Theorem~\ref{thm:main2} follows from Theorem~\ref{thm:main} and the following fact (see~\cite{kyncl0,pt}): If $\mb{G}$ and $\mb{H}$ are rotation systems of complete topological graphs $G$ and $H$, respectively, and $\mb{G}{=}\mb{H}$ or $\mb{G}{=}\mb{H}^{-1}$, then $G$ and $H$ are weakly isomorphic. However, Pach et al.~proved Theorem~\ref{thm:main2} with the bound $n=2^{m^8}$, and an explicit bound for $n$ using Theorem~\ref{thm:main} would be multiply exponential in $m$, as an explicit bound for $n_0$ in Theorem~\ref{thm:main} using our arguments is multiply exponential in $m$. What is the best bound for $n_0$ in Theorem~\ref{thm:main}? Is it singly exponential in $m$? %Can one possibly establish such a bound 

%%% We have an alternative proof of Theorem~\ref{thm:main}, but it also gives a ,

%%% 

% *****************************************************************************
% *****************************************************************************
% *************************    *********************************
% *****************************************************************************
% *****************************************************************************

% *****************************************************************************
% *****************************************************************************
% *************************  ACKNOWLEDGMENTS  *********************************
% *****************************************************************************
% *****************************************************************************

\section*{Acknowledgements}

The second author was supported by NSERC Grant No.~50503-10940-500. This project has received funding from the European Union's Horizon 2020 research and innovation programme under the Marie Sk\l{}odowska-Curie grant agreement No 754411.

% *****************************************************************************
% *****************************************************************************
% ***************************  REFERENCES  ************************************
% *****************************************************************************
% *****************************************************************************

% *****************************************************************************
% *****************************************************************************
% *****************************************************************************
% *****************************************************************************
% *****************************************************************************

\end{document}